\documentclass[12pt]{article}

\usepackage{latexsym}
\usepackage{amsfonts}
\usepackage{amssymb}
\usepackage{amsmath}
\usepackage{longtable}
\usepackage{mathrsfs}
\usepackage{colordvi}
\usepackage{amsthm}
\usepackage{graphicx}
\usepackage{verbatim}
\usepackage[unicode,final,hyperindex]{hyperref}
\usepackage{bbm}
\usepackage[mathlines]{lineno}
\usepackage{cite}
\usepackage[varg]{txfonts}
\usepackage{stmaryrd}

\newcommand*\patchAmsMathEnvironmentForLineno[1]{%
  \expandafter\let\csname old#1\expandafter\endcsname\csname #1\endcsname
  \expandafter\let\csname oldend#1\expandafter\endcsname\csname end#1\endcsname
  \renewenvironment{#1}%
     {\linenomath\csname old#1\endcsname}%
     {\csname oldend#1\endcsname\endlinenomath}}%
\newcommand*\patchBothAmsMathEnvironmentsForLineno[1]{%
  \patchAmsMathEnvironmentForLineno{#1}%
  \patchAmsMathEnvironmentForLineno{#1*}}%
\AtBeginDocument{%
\patchBothAmsMathEnvironmentsForLineno{equation}%
\patchBothAmsMathEnvironmentsForLineno{align}%
\patchBothAmsMathEnvironmentsForLineno{flalign}%
\patchBothAmsMathEnvironmentsForLineno{alignat}%
\patchBothAmsMathEnvironmentsForLineno{gather}%
\patchBothAmsMathEnvironmentsForLineno{multline}%
}

\renewcommand{\le}{\leqslant}
\renewcommand{\ge}{\geqslant}
\renewcommand{\leq}{\leqslant}
\renewcommand{\geq}{\geqslant}
\renewcommand{\unlhd}{\trianglelefteqslant}

 \DeclareMathOperator{\Aut}{Aut}
\DeclareMathOperator{\Sym}{Sym} \DeclareMathOperator{\Inn}{Inn}

\DeclareMathOperator{\Alt}{Alt}
\DeclareMathOperator{\Base}{Base}\DeclareMathOperator{\Reg}{Reg}
\newcommand{\splitext}{\,\colon\!}

\begin{document}

\newtheorem{lem}{Lemma}
\newtheorem{thm}[lem]{Theorem}
\newtheorem{cor}[lem]{Corollary}
\newtheorem{prop}[lem]{Proposition}
\theoremstyle{remark}
\newtheorem{exmp}[lem]{Example}
\newtheorem{prob}{Problem}


\begin{flushright}
 MSC 20D60, 20B40, 20B35
\end{flushright}

\begin{center}
{\bf ON THE BASE SIZE OF A TRANSITIVE GROUP WITH SOLVABLE POINT STABILIZER\footnote{The author gratefully acknowledges
the support from  RFBR, projects 10-01-00391 and 10-01-90007,  ADTP
``Development
of the Scientific Potential of Higher
School'' of the Russian Federal Agency for Education (Grant
2.1.1.419),  Federal Target Grant "Scientific and
educational personnel of innovation Russia" for 2009-2013 (government
contract No. 02.740.11.5191 and No. 14.740.11.0346), Deligne 2004 Balzan prize
in mathematics, and the Lavrent'ev Young Scientists Competition (No 43 on
04.02.2010).}\\[2ex]
{\bf E.\,P.\,Vdovin}}\\[3ex]
\end{center}

\begin{small}

\quad We prove that the base size of a transitive group $G$ with solvable point
stabilizer and with trivial solvable radical is not greater than $k$
provided the same statement holds for the group of $G$-induced automorphisms of
each nonabelian composition factor of
$G$.

\end{small}

\section{Introduction}

The term ``group'' always means a ``finite group''.
We use symbols $A\leq G$ and $A\unlhd G$ if  $A$ is a
subgroup of
$G$, and $A$ is a
normal subgroup of $G$, respectively. If $\Omega$ is a (finite) set, then by $\Sym(\Omega)$ we denote the group of all
permutations of $\Omega$. We also denote $\Sym(\{1,\ldots,n\})$ by $\Sym_n$.
 Given $H\leq G$ we denote by $H_G=\cap_{g\in G} H^g$ the core of~$H$.

Let $A,B$ be subgroups of $G$ such that $B\unlhd A$. Then $N_G(A/B):=N_G(A)\cap
N_G(B)$ is the {\em normalizer} of
$A/B$ in $G$. If $x\in
N_G(A/B)$, then $x$ induces an automorphism of $A/B$ by $Ba\mapsto B x^{-1}ax$.
Thus there exists a homomorphism $N_G(A/B)\rightarrow \Aut(A/B)$. The image of
$N_G(A/B)$ under this homomorphism is denoted by $\Aut_G(A/B)$ and is called a
{\em group of $G$-induced automorphisms of} $A/B$.

Assume that $G$ acts on $\Omega$. An element $x\in
\Omega$ is called a {\em
$G$-regular point} if $\vert
xG\vert=\vert G\vert$, i.e., if the
$G$-orbit of $x$ is regular. Define an action of $G$ on $\Omega^k$ by
\begin{equation*}
g:(i_1,\ldots,i_k)\mapsto
(i_1g,\ldots,i_kg).
\end{equation*}
If $G$ acts faithfully and transitively on $\Omega$, then the minimal $k$ such that $\Omega^k$
possesses a $G$-regular
point is called the {\em base size} of $G$ and is denoted by~$\Base(G)$. For
every natural $m$ the number of $G$-regular
orbits in $\Omega^m$ is
denoted by $\Reg(G,m)$ (this number equals $0$ if $m<\Base(G)$). If $H$ is a
subgroup of $G$ and $G$ acts on the set $\Omega$
of right cosets of $H$ by right multiplications, then
$G/H_G$ acts faithfully and  transitively on $\Omega$. In this case we denote $\Base(G/H_G)$ and $\Reg(G/H_G,m)$ by
$\Base_H(G)$ and
$\Reg_H(G,m)$ respectively. We also say that $\Base_H(G)$ is the {\em base size
of} $G$ {\em with respect to}~$H$. Clearly, $\Base_H(G)$ is the minimal $k$
such that there exist $x_1,\ldots,x_k\in G$ with $H^{x_1}\cap \ldots\cap
H^{x_k}=H_G$.

There are a lot of papers dedicated to this subject.  It is impossible to
mention all of them, since the list of references would be much longer that the
paper.  We mention the papers, whose results are used in the present article.
A.Seress in \cite[Theorem~1.1]{Seress} proved that the base size of a primitive
solvable permutation group is not greater than $4$.
In \cite{Dolfi} S.Dolfi proved that in every $\pi$-solvable group $G$ there exist elements
$x,y\in G$ such that
the equality  $H\cap H^x\cap H^y=O_\pi(G)$ holds, where $H$ is a $\pi$-Hall subgroup of $G$ (see
also~\cite{VdovinIntersSolv}). V.I.Zenkov in \cite{Zenkov}
constructed an example of a group $G$ with a solvable $\pi$-Hall subgroup $H$
such that the intersection of five subgroups conjugate with $H$ in $G$ is equal to $O_\pi(G)$, while the intersection
of every four conjugates of $H$ is greater than $O_\pi(G)$ (see Example
\ref{S5wrS2} below). In \cite{VdovinZenkov} it is proven that if, for
every
almost simple group $S$ possessing a solvable $\pi$-Hall subgroup $H$, the
inequalities $\Base_H(S)\leq 5$ and $\Reg_H(S,5)\geq5$ hold, then
for every group $G$ possessing a
solvable $\pi$-Hall subgroup $H$  the
inequality $\Base_H(G)\leq 5$ holds. In
the present paper we  prove the following

\begin{thm}\label{MainTheorem}
Let $G$ be a group and let
\begin{equation}\label{compositionseries}
\{e\}=G_0< G_1< G_2<\ldots <G_n=G
\end{equation}
be a composition series of $G$ that is a refinement of a chief
series.  Assume that for some $k$ the following condition
{\em\bfseries(Orb-solv)} holds: If $G_i/G_{i-1}$ is nonabelian, then  for
every solvable subgroup $T$ of $\Aut_G(G_i/G_{i-1})$
we have
\begin{equation*}
 \Base_{T}(T(G_i/G_{i-1}))\le k\text{  and }
\Reg_{T}(T(G_i/G_{i-1}),k)\ge 5.
\end{equation*}
Then, for every maximal solvable subgroup $S$ of $G$, we have
$\Base_S(G)\le k$.
\end{thm}

The author of the paper insert to
the ``Kourovka notebook'' \cite{Kour} the following problem 17.41.

\begin{prob}\label{mainprob}
 Let $S$ be a solvable subgroup of a group $G$ with $S(G)=\{e\}$.
\begin{itemize}
 \item[(a)] (L.Babai, A.J.Goodman, L.Pyber) Do there exist $7$ conjugates of
$S$ such that their intersection is trivial?
\item[(b)]  Do  there exist $5$ conjugates of $S$ such that their intersection
is
trivial?
\end{itemize}
\end{prob}

Theorem \ref{MainTheorem} reduces both parts of Problem \ref{mainprob} to
the investigation of
almost simple groups.  Notice also that Theorem \ref{MainTheorem} generalizes
the main
result of \cite{VdovinZenkov} in the following way.

\begin{cor}\label{VdoZenkovGeneralization}
 Let $G$ be a  group possessing a solvable $\pi$-Hall subgroup $H$. Assume that
for $k=5$ condition
{\em\bfseries(Orb-solv)} holds.
Then $\Base_H(G)\le 5$.
\end{cor}

We prove the corollary in Section 3 of the article.

We remark that recently it was proved by T.C.Burness, M.W.Liebeck, E.O'Brien,
A.Shalev, R.A.Wil\-son, etc that if $G$ is a
primitive group of almost simple type and the action is not standard, then $G$
has the base size at most $7$, answering a conjecture of Peter Cameron (see
\cite{BurnessLiebeckShalev} and the bibliography thereafter). In light of
Theorem~\ref{MainTheorem}, these
results seem to be relevant to a solution of Problem \ref{mainprob} in finite
almost simple groups. Nevertheless they cannot be applied immediately since
arbitrary solvable subgroup of a
symmetric group or of a classical group may lie in a maximal subgroup
giving a standard action.

\section{Notation and preliminary results}

By $\vert G\vert$  we denote the cardinality of $G$.
By $A\splitext B$ we denote a split extension of a group $A$ by a group $B$.
For a group $G$ and a
subgroup $M$ of $\Sym_n$ by $G\wr M$ we always denote the permutation wreath
product. We identify $G\wr M$ with the natural split extension $(G_1\times\ldots\times G_n)\splitext M$, where
$G_1\simeq\ldots\simeq G_m\simeq G$ and $M$ permutes $G_1,\ldots,G_n$. Given
group $G$, we denote by  $S(G)$
the maximal
normal solvable subgroup of $G$. We denote
by $e$ the identity element of~$G$. A group  $G$ is called {\em almost simple}
if there exists a nonabelian simple group $L$ such that $\Inn(L)\leq
G\leq\Aut(L)$.

The following statement is evident.

\begin{lem}\label{maximalsolvable}
If $S$ is a maximal solvable subgroup  of $G$, then $N_G(S)=S$.
\end{lem}

\begin{lem}\label{InducedAutomorphismsUnderHomomorphism} {\em \cite[Lemma~1.2]{VdoAlmSimp}}
Let $H$ be a normal subgroup of a group $G$, and let
$(A/H)/(B/H)$ be a composition factor of~$G/H$.

Then $\Aut_G(A/B)\simeq \Aut_{G/H}((A/H)/(B/H))$.
\end{lem}

\begin{lem}\label{NormalizerIntersection}
 Let $S$ be a maximal solvable subgroup of $G$ and let $N$ be a normal subgroup of $G$ containing $S(G)$. Then $N_N(N\cap S)=N\cap S$.
\end{lem}

\begin{proof}
Assume that the claim is false and $G$ is a counter example of minimal order. Assume that $S(G)\not=\{e\}$ and
consider the natural
homomorphism
\begin{equation*}
\overline{\phantom{G}}:G\rightarrow G/S(G).
\end{equation*}
Clearly $\overline{S}$ is a maximal solvable subgroup
of $\overline{G}$ and
$S(\overline{G})=\overline{S(G)}=\{e\}$. Moreover, $\vert \overline{G}\vert<\vert G\vert$. Since $G$ is a counter example of minimal order
it follows that
$N_{\overline{N}}(\overline{N}\cap\overline{S})=\overline{N}\cap\overline{S}$. Now $S(G)$ lies in both $N$ and $S$, hence $N_N(N\cap S)$ is
a complete preimage of $N_{\overline{N}}(\overline{N}\cap\overline{S})=\overline{N}\cap\overline{S}$, and so $N_N(N\cap S)=N\cap S$. Thus
$S(G)=\{e\}$.

Set $M=N_G(N\cap S)$, $L=N_N(N\cap S)=N\cap M$. In view of \cite[Proposition~3]{GurWil}, $N\cap S\not=\{e\}$, so
$S(M)\ge S\cap N\not=\{e\}$ and $M$ is a proper subgroup
of $G$. Clearly $S\leq M$, so the maximality of $S$ implies $S(M)\leq  S$.
Moreover $L$ is normal in $M$. So $LS(M)$ is normal in $M$. Since $\vert
M\vert<\vert
G\vert$, we obtain
\begin{equation*}
N_{LS(M)}(S\cap LS(M)))=S\cap LS(M)=(S\cap L)S(M)\leq S.
\end{equation*}
 Now suppose that $x\in L$. We have $N\cap S\leq L\leq N$, so  $L\cap S=N\cap
S$. By construction,  $L=N_N(L\cap S)$, so $L\cap S\unlhd L$. Moreover $L\leq
M$, hence $x$ normalizes $S(M)$, and so $x$ normalizes
$(S\cap L)S(M)=N_{LS(M)}(S\cap LS(M)))$, in particular, $x\in
S$. Thus~${L=S\cap N}$ and $G$ is not a counter example.
\end{proof}

Assume that a group $G$ possesses a normal subgroup
$T$  satisfying the following conditions:

\begin{itemize}
\item[(C1)]  there exists a nonabelian simple group $L$ such that $T\simeq
L_1\times
\ldots\times
L_k$ and $L_1\simeq\ldots\simeq L_k\simeq L$;
\item[(C2)] the subgroups $L_1,\ldots,L_k$ are conjugate in $G$;
\item[(C3)] $C_G(T)=\{e\}$.
\end{itemize}

By \cite[Satz~12.5, p.~69]{Huppert}, $G$ acting by conjugation on $T$ permutes
$L_1,\ldots,L_k$.
Condition~(C2) implies that  $N_G(L_1),\ldots,N_G(L_k)$ are conjugate in $G$. It
follows  that $G$ acts on the right cosets of $N_G(L_1)$ by
right multiplication, let $\rho:G\rightarrow \Sym_k$ be the corresponding
permutation representation. The action by right
multiplication of $G$ on the right cosets of
$N_G(L_1)$ coincides with the action by conjugation of $G$ on the set
$\{L_1,\ldots,L_k\}$, and $G\rho$ is a transitive subgroup of
 $\Sym_k$. By \cite[Hauptsatz~1.4, p.~413]{Huppert}
there exists a monomorphism
\begin{equation*}
\varphi:G\rightarrow (N_G(L_1)\times \ldots\times N_G(L_k)): (G\rho)=N_G(L_1)\wr
(G\rho)=M.
\end{equation*}
Since $C_G(L_i)$ is a normal subgroup of $N_G(L_i)$, it follows that
$C_G(L_1)\times\ldots\times C_G(L_k)$ is a normal
subgroup of $M$.
Consider the
natural homomorphism
\begin{equation*}
 \psi:M\rightarrow M/(C_G(L_1)\times\ldots\times C_G(L_k)).
\end{equation*}
 Denoting
$\Aut_G(L_i)=N_G(L_i)/C_G(L_i)$ by $A_i$ we obtain a homomorphism
\begin{equation*}
 \varphi\circ\psi:G\rightarrow (A_1\times\ldots\times
A_k):(G\rho)\simeq A_1\wr (G\rho)=:\overline{G}.
\end{equation*}
As $C_G(T)=\{e\}$, the kernel
of $\varphi\circ\psi$ is equal to  $C_G(L_1,\ldots,L_k)=\{e\}$,
i.~e.,
$\varphi\circ\psi$ is a monomorphism and we
identify  $G$ and subgroups of $G$ with
their images under~$\varphi\circ\psi$.

\begin{lem}\label{NonAbelianSimpleDirectAutomorphisms}
Assume that $G$ possesses a normal subgroup $T$ satisfying conditions
{\em (C1), (C2), and (C3)}. Assume also that $G/T$ is
solvable. Consider the monomorphism  $\varphi\circ\psi$ defined above. Then the
followings hold:
\begin{itemize}
 \item[{\em (a)}] there exists a maximal solvable subgroup $S$ of $G$ such that
$G=ST$;
\item[{\em (b)}] if we choose a maximal solvable subgroup $S$ of $G$ such that
$G=ST$, then $\overline{G}$ possesses a maximal solvable subgroup
$\overline{S}$ such that $S\leq \overline{S}$ and $\overline{G}=\overline{S} T$.
\end{itemize}
\end{lem}

\begin{proof}
(a) Consider a minimal subgroup $M$ of $G$ such that $G=MT$. Clearly $M\cap
T$ is normal in $M$ and is included in the Frattini subgroup $\Phi(M)$ of $M$.
Otherwise $M$
possesses a
proper subgroup $M_1$ such that
$M_1(M\cap T)=M$ and so $G=M_1T$, a contradiction with the minimality of $M$.
Since $\Phi(M)$ is nilpotent and $M/(M\cap T)$ is solvable, it follows that $M$
is solvable. Let $S$ be a
maximal solvable subgroup of $G$ containing $M$, then $G=ST$.

(b) Condition (C2) implies $A_i=\Aut_{\overline{G}}(L_i)=\Aut_G(L_i)\simeq
\Aut_G(L_1)$ for all $i$. Since $[L_i,L_j]=\{e\}$ for $i\not=j$ and $G=ST$, we
obtain that
\begin{equation*}
A_i=\Aut_G(L_i)=N_G(L_i)/C_G(L_i)= N_S(L_i)T/C_G(L_i),
\end{equation*}
 and so $A_i/L_i\simeq N_S(L_i)/(N_S(L_i)\cap L_iC_G(L_i))$ is
solvable. Therefore $\overline{G}/(L_1\times\ldots\times L_n)\simeq (A_1/L_1)\wr
(G\rho)$ is solvable. Consider $H=S\cap T$ and denote by $\pi_i$ the natural
projection $L_1\times\ldots\times L_k\rightarrow L_i$. Put $H_i=H^{\pi_i}$.
Clearly, $H\leq H_1\times \ldots\times H_k$. If $x\in S$ and $L_i^x=L_j$, then
$H_i^x=H_j$, since $H$ is normal in $S$. Hence $S$ normalizes $H_1\times
\ldots\times H_k$, and by the maximality of
$S$ we have $S\ge H_1\times\ldots\times H_k$, i.e., $H=H_1\times\ldots\times
H_k$. Clearly
\begin{equation*}
N_T(H)=N_{L_1\times\ldots\times L_k}(H_1\times\ldots\times
H_k)=N_{L_1}(H_1)\times\ldots\times N_{L_k}(H_k).
\end{equation*} By Lemma \ref{NormalizerIntersection} we have $N_T(H)=H$, so
$N_{L_i}(H_i)=H_i$ for $i=1,\ldots,k$. As $N_S(L_i)\leq N_{A_i}(H_i)$, it
follows that
$A_i$ is equal to $N_{A_i}(H_i)L_i$ and $N_{A_i}(H_i)$ is solvable.  We obtain
that
\begin{equation*}
A_1\times\ldots\times A_k=(N_{A_1}(H_1)\times\ldots\times
N_{A_k}(H_k))T=N_{A_1\times\ldots\times A_k}(H)T
\end{equation*}
 and $N_{A_1\times\ldots\times A_k}(H)$
is solvable. Since
$\overline{G}=(A_1\times\ldots\times A_k)S$, and since $S$ normalizes $H$, it
follows that
$S$ lies in $N_{\overline{G}}(H)$, and so $\overline{G}=N_{\overline{G}}(H)T$.
Moreover $N_{\overline{G}}(H)=N_{A_1\times\ldots\times A_k}(H)$ is solvable,
therefore there exists a maximal
solvable subgroup $\overline{S}$ of $\overline{G}$, containing
$N_{\overline{G}}(H)$. Thus we obtain that  $S\leq \overline{S}$ and
$\overline{G}=\overline{S}T$.
\end{proof}

Let  $G$ be a subgroup of $\Sym_n$. A partition $P_1\sqcup P_2\sqcup\ldots\sqcup
P_m$ of $\{1,\ldots,n\}$ is called an
{\em asymmetric partition} for  $G$, if only the identity element of $G$ fixes the partition, i.~e., the equality $P_jx=P_j$ for all
$j=1,\ldots,m$ implies  $x=e$. Clearly for every $G$  the partition $P_1=\{1\},
P_2=\{2\},\ldots,P_n=\{n\}$ is  asymmetric.

\begin{lem}{\rm \cite[Theorem~1.2]{Seress}}
\label{Seresslemm}
Let $G$ be a solvable subgroup of $\Sym_n$.  Then there exists an asymmetric
partition $P_1\sqcup P_2 \sqcup\ldots\sqcup
P_m=\{1,\ldots,n \}$ with~${m\le 5}$.
\end{lem}

\begin{lem}\label{WreathWithSolvableBase}
 Let $G$ be a  group and let $M$ be a solvable subgroup of $\Sym_n$. Assume
that there exists $k$ such that
for every maximal solvable
subgroup $T$ of $G$ the inequalities
\begin{equation*}
 \Base_T(G)\leq k\text{ and }\Reg_T(G,k)=s\ge5
\end{equation*}
hold. Then for every maximal solvable subgroup $S$ of $G\wr M$ we have
$\Base_S(G\wr M)\le k$.
\end{lem}

\begin{proof}
 We have $G\wr M=\left(G_1\times\ldots\times G_n\right)\splitext M$. Moreover $S(G\wr M)=S(G_1)\times
\ldots\times S(G_n)$, since $C_M(G_1\times \ldots\times G_n)=\{e\}$. Assume by contradiction that $G\wr M$ is a counter
example to the lemma with $\vert G\wr M\vert $ minimal. Then clearly $S(G\wr M)=\{e\}$, i.e., $S(G)=\{e\}$, otherwise
we substitute $G$ by $G/S(G)$ and proceed by induction.

Since $G\wr M$ is a counter example to the lemma, there exists a maximal solvable subgroup $S$ of $G\wr M$ such that
for every $x_1,\ldots,x_k\in G\wr M$ we have $S^{x_1}\cap\ldots\cap S^{x_k}\not=\{e\}$. It is clear that $(G_1\times
\ldots\times G_n)S=G\wr M$, otherwise consider the image $\overline{S}$ of $S$ under the natural homomorphism $G\wr
M\rightarrow M$. We obtain that $(G_1\times
\ldots\times
G_n)S= G\wr \overline{S}< G\wr M$, so  we substitute $G\wr M$ by $G\wr
\overline{S}$ and
proceed by induction. The
minimality of $G\wr M$  implies also that $M$ is
transitive, otherwise we would obtain that
$G\wr M\leq (G\wr M_1)\times (G\wr M_2)$, where $M_1\leq \Sym_m$, $M_2\leq
\Sym_{n-m}$, and proceed by induction. Indeed denote
the projections of $G\wr M$ onto $G\wr M_1$ and
$G\wr M_2$  by $\pi_1$ and $\pi_2$  respectively. Up to renumbering we may
suppose that  $G\wr M_1=(G_1\times\ldots\times
G_m)\splitext M_1$ and $G\wr M_2=(G_{m+1}\times\ldots\times
G_n)\splitext M_2$. Denote $G_1\times\ldots\times G_m$ by $E_1$ and
$G_{m+1}\times\ldots\times G_n$ by $E_2$. Since $G\wr
M=(G_1\times\ldots \times G_n)S$, $E_1\leq
\mathrm{Ker}(\pi_2)$, and $E_2 \leq
\mathrm{Ker}(\pi_1)$, it
follows that
$(G\wr M)\pi_i= E_i(S\pi_i)$ (we identify
$E_i\pi_i$ with $E_i$, since $E_i\pi_i\simeq
E_i$).
By induction for each $i\in\{1,2\}$ there
exist elements $x_{1,i},\ldots,x_{k,i}$ of
$E_i(S\pi_i)$ such that
\begin{equation}\label{Proj1}
(S\pi_i)^{x_{1,i}}\cap \ldots\cap
(S\pi_i)^{x_{k,i}}=\{e\}.
\end{equation}
Since $G\pi_i=E_i(S\pi_i)$, we may assume that $x_{1,i},\ldots,x_{k,i}$ are in
$E_i$. Consider
$x_1=x_{1,1}x_{1,2},\ldots,x_k=x_{k,1}x_{k,2}$.
Since \eqref{Proj1} is true for every $i$,  we
have
\begin{equation*}
S^{x_1}\cap \ldots\cap S^{x_k}=\{e\},
\end{equation*}
and $G$ is not a counter example.

Consider $L=S\cap (G_1\times\ldots\times G_n)$ and denote by $\pi_i$ the natural
projection $G_1\times\ldots\times
G_n\rightarrow G_i$. Put $L_i=L^{\pi_i}$. Clearly $L\leq L_1\times \ldots\times L_n$. If $x\in S$ and $G_i^x=G_j$, then
$L_i^x=L_j$, since $L$ is normal in $S$. Hence $S$ normalizes
$L_1\times\ldots\times L_n$ and by the maximality of~$S$ we have
$L=L_1\times\ldots\times L_n$.

Clearly $N_{G_1\times\ldots\times G_n}(L_1\times\ldots\times L_n)=N_{G_1}(L_1)\times\ldots\times N_{G_n}(L_n)$. By
Lemma \ref{NormalizerIntersection} we obtain that $N_{G_1\times\ldots\times G_n}(L_1\times\ldots\times
L_n)=L_1\times\ldots\times L_n$, hence $N_{G_i}(L_i)=L_i$ for $i=1,\ldots,n$.
Denote by $\Omega_i$ the set $\{L_i^x\mid
x\in G_i\}$, then $G_i$ acts on $\Omega_i$ by conjugation. Since $N_{G_i}(L_i)=L_i$, it follows that $L_i$ is the point
stabilizer under this action. Set $\Omega=\Omega_1\times\ldots\times \Omega_n$. For every $x\in G\wr
M$ and for every $i$ we have $L_i^x\leq G_j$ for some $j$. We show that
\begin{equation}\label{orbitGi}
\text{if } L_i^x\leq G_j\text{ then } L_i^x\in L_j^{G_j}\text{, i.e., there
exists }y\in
G_j\text{ such that }L_j^y=L_i^x.
\end{equation}
 Since $(G_1\times\ldots\times G_n)\splitext M=(G_1\times\ldots\times G_n)S$, it
follows
that there exists $s\in S$ with $G_i^s=G_j$. We also have $L_i^s=L_j$, since $L$ is normal in $S$. Thus
$L_i^x=L_j^{s^{-1}x}$. Now $s^{-1}x=g_1\cdot \ldots\cdot g_n\cdot h$, where $g_i\in G_i$ for $i=1,\ldots,n$ and $h\in
M$. Since $M$ permutes the $G_i$-s, it follows that for every
$i=1,\ldots,n$, either $G_i^h\cap G_i=\{e\}$, or
$h$ centralizes $G_i$. Thus we
obtain that $L_j^{s^{-1}x}=L_j^{g_j}$. So $G\wr M$ acts by conjugation on
$\Omega$ and  $S$ is the stabilizer of the
point $(L_1,\ldots,L_n)$. Therefore we need to show that $\Omega^k$ possesses
a $(G\wr
M)$-regular orbit.

The conditions of the
lemma imply that there exist $G_1$-regular points
$\omega_1,\ldots,\omega_s\in\Omega_1^k$ lying in distinct $G_1$-orbits. If we
choose $h_1=e, h_2,\ldots,h_n\in M$ so that $G_1^{h_i}=G_i$, then $\omega_1^{h_i},\ldots,\omega_s^{h_i}\in\Omega_i^k$
are $G_i$-regular points, and \eqref{orbitGi} implies that they are in distinct
$G_i$-orbits. We set
$\omega_{i,j}=\omega_i^{h_j}$. By Lemma \ref{Seresslemm} there exists an asymmetric partition $P_1\sqcup P_2\sqcup
P_3\sqcup P_4\sqcup P_5=\{1,\ldots,n\}$ for $M$. Since $s\ge 5$ we can choose
$\omega=(\omega_{i_1,1},\ldots,\omega_{i_n,n})$ so that
$i_t=i_j$ if and only if $t,j$ lie in the same $P_l$. Now we show that
$\omega\in \Omega^k$ is a $(G\wr
M)$-regular point.
Indeed, consider $g=(g_1\ldots g_n)h$, where $g_i\in G_i$ for $i=1,\ldots,n$ and
$h\in M$, and assume that $\omega
g=\omega$. It follows that $\omega h^{-1}=\omega  (g_1\ldots g_n)$, i.e.,
\begin{equation*}
(\omega_{i_1,1},\ldots,\omega_{i_n,n})h^{-1}=
(\omega_{i_{(1h)},1},
\ldots , \omega_ { i_{(nh)} ,
n})=(\omega_{i_1,1}g_1,\ldots,\omega_{i_n,n}g_n).
\end{equation*}
Therefore  $\omega_{i_{(jh)},j}$ and $\omega_{i_j,j}$ are in the
same $G_j$-orbit, i.e., $i_{(jh)}=i_j$. By construction, $jh$ and $j$
are in the same $P_l$. Whence $h$ stabilizes the partition $P_1\sqcup
P_2\sqcup P_3\sqcup P_4\sqcup P_5$ and $h=e$. We obtain that
$(\omega_{i_1,1},\ldots,\omega_{i_n,n})=
(\omega_{i_1,1}g_1,\ldots,\omega_{i_n,n}g_n)$. By construction,
$\omega_{i_j,j}$ is a $G_j$-regular point for every $j=1,\ldots,n$, so
$g_1=\ldots=g_n=e$, i.e., $g=e$ and $\omega\in\Omega^k$ is a $(G\wr M)$-regular
point.
\end{proof}

\section{Proof of the main theorem and the corollary}

\noindent {\itshape Proof of Theorem {\em\ref{MainTheorem}}.}\ \  Assume that
the claim is
false and $G$ is a counter example of minimal order. Fix a maximal solvable
subgroup $S$ of $G$ with $\Base_S(G)>k$.

Assume that $S(G)\not=\{e\}$. Then there exists a minimal elementary abelian
normal
subgroup $K$ of $G$. Since elements from distinct minimal normal
subgroups commute, we may suppose that $G_1\leq K$ and there exists $l$ such
that $G_l=K$, i.e., the composition series
\eqref{compositionseries} is a
refinement of a chief series starting with $K$. In this case, if
\begin{equation*}
 \overline{\phantom{G}}:G\rightarrow G/K=\overline{G}
\end{equation*}
is the
natural homomorphism, then
\begin{equation*}
\{\bar{e}\} =\overline{G}_l< \overline{G}_{l+1}<\ldots<\overline{G}_n=\overline{G}
\end{equation*}
is a composition series
of $\overline{G}$ that is a refinement of a chief series of $\overline{G}$. Moreover, for every nonabelian
$\overline{G}_i/\overline{G}_{i-1}$, Lemma \ref{InducedAutomorphismsUnderHomomorphism} implies
$\Aut_{\overline{G}}(\overline{G}_i/\overline{G}_{i-1})\simeq
\Aut_G(G_i/G_{i-1})$. Since $G$ satisfies {\bfseries(Orb-solv)} for some $k$,
we obtain that $\overline{G}$ satisfies
{\bfseries(Orb-solv)} for the same $k$.   In view of the minimality of $G$,
there exist
$x_1,\ldots,x_k\in G$ such that
\begin{equation*}
\overline{S}^{\bar{x}_1}\cap \ldots\cap
\overline{S}^{\bar{x}_k}=S(\overline{G}).
\end{equation*}
 Now $K\leq S(G)$, hence
$\overline{S(G)}=S(\overline{G})$. Therefore  $S^{x_1}\cap \ldots\cap
S^{x_k}=S(G)$, i.e.,
$G$ is not a counter
example.

Thus we may assume that $S(G)=\{e\}$. Consider the generalized Fitting subgroup $F^\ast (G)$ of $G$. Since $S(G)=\{e\}$,
we obtain that
$F^\ast(G)=L_1\times\ldots\times L_n$ is a product of nonabelian simple
groups and, by \cite[Theorem~9.8]{Isaacs},
$C_G(F^\ast
(G))=Z(F^\ast(G))=\{e\}$. In particular, $S(F^\ast(G)S)=\{e\}$. If
$F^\ast(G)S\not= G$, then, in view of the minimality of $G$, there exist
$x_1,\ldots,x_k\in F^\ast(G)S$ such that $S^{x_1}\cap \ldots\cap
S^{x_k}=S(F^\ast(G)S)=\{e\}$, i.e., $G$ is not a
counter example. It follows that $G=F^\ast(G)S$. Moreover, since
$L_1,\ldots,L_n$ are nonabelian simple, \cite[Satz~12.5, p.~69]{Huppert} implies
that
$G$, acting by conjugation, permutes the elements of~${\{L_1,\ldots,L_n\}}$.

Set $E_1:=\langle L_1^S\rangle$ and $E_2=\langle L_i\mid
L_i\not\in\{L_1^s\mid s\in S\}\rangle$. Since $F^\ast(G)=L_1\times\ldots\times
L_n$, we
obtain that  $F^\ast(G)=E_1\times E_2$, where
$E_1$ and $E_2$ are $S$-invariant subgroups. By \cite[Hilfssatz~9.6,
p.~48]{Huppert} there exists a homomorphism
$G\rightarrow G/C_G(E_1)\times G/C_G(E_2)$, such that the image of $G$ is a subdirect product of $G/C_G(E_1)$ and $G/C_G(E_2)$, while the
kernel is equal to
$C_G(E_1)\cap C_G(E_2)=C_G(F^\ast(G))=\{e\}$. Denote the projections of $G$ onto
$G/C_G(E_1)$ and
$G/C_G(E_2)$  by $\pi_1$ and $\pi_2$  respectively. Since $G=F^\ast(G)S$,
$E_1\leq \mathrm{Ker}(\pi_2)$ and $E_2\leq \mathrm{Ker}(\pi_1)$, it
follows that
$G\pi_1=E_1 (S\pi_1)$ and $G\pi_2=E_2 (S\pi_2)$  (we identify $E_i\pi_i$ with
$E_i$ since $E_i\pi_i\simeq E_i$).

Suppose that $E_1\ne F^\ast(G)$. Then, by induction for each $i\in\{1,2\}$ there
exist elements $x_{1,i},\ldots,x_{k,i}$ of
$E_i(S\pi_i)$ such that
\begin{equation}\label{Proj}
(S\pi_i)^{x_{1,i}}\cap \ldots\cap
(S\pi_i)^{x_{k,i}}=\{e\}.
\end{equation}
Since $G\pi_i=E_i(S\pi_i)$, we may assume that $x_{1,i},\ldots,x_{k,i}$ are in $E_i$. Consider
$x_1=x_{1,1}x_{1,2},\ldots,x_k=x_{k,1}x_{k,2}$.
Since \eqref{Proj} is true for every $i$ and $\mathrm{Ker}(\pi_1)\cap
\mathrm{Ker}(\pi_2)=\{e\}$,   we have
\begin{equation*}
S^{x_1}\cap \ldots\cap S^{x_k}=\{e\},
\end{equation*}
and $G$ is not a counter example.

Therefore  $E_1=F^\ast(G)$ and $S$ acts transitively on $\{L_1,\ldots,L_n\}$.
Since $\Aut_G(L_1)$ satisfies
{\bfseries(Orb-solv)} for some $k$, we may assume
that $n>1$. By Lemma \ref{NonAbelianSimpleDirectAutomorphisms} and by the
discussion preceding it, we may assume
that $G=(A_1\times\ldots\times
A_n)\splitext K=A_1\wr
K$, where $A_i=\Aut_G(L_i)$, $K=G\rho\leq\Sym_n$ and $\rho$ is the permutation
representation of $G$ on the set $\{L_1,\ldots,L_n\}$. Since $G=F^\ast(G)S$, we
see that $K=S\rho$ is solvable. Lemma
\ref{WreathWithSolvableBase} (applied with $G=A$) implies that $\Base_S(G)\le k$
for every maximal solvable subgroup $S$ of $G$. This final
contradiction completes the proof. \qed
\vspace{\baselineskip}

\noindent {\itshape Proof of Corollary {\em\ref{VdoZenkovGeneralization}}.}\ \
Let $G$ be a group
satisfying {\bfseries(Orb-solv)} for $k=5$. Assume that $G$ possesses a solvable
$\pi$-Hall subgroup $H$. Consider the natural homomorphism
\begin{equation*}
\overline{\phantom{G}}:G\rightarrow G/S(G).
\end{equation*}
 Since $H$ is
solvable, it follows that there exists a maximal solvable subgroup $S$ of $G$
with $H\leq S$. By Theorem \ref{MainTheorem} there exist $x_1,x_2,x_3,x_4,x_5$
such that
\begin{equation*}
S^{x_1}\cap S^{x_2}\cap S^{x_3}\cap S^{x_4}\cap S^{x_5}=S(G).
\end{equation*}
Thence $H^{x_1}\cap H^{x_2}\cap H^{x_3}\cap H^{x_4}\cap H^{x_5}\leq S(G)$ and
$\overline{H}^{\bar{x}_1}\cap \overline{H}^{\bar{x}_2}\cap
\overline{H}^{\bar{x}_3}\cap \overline{H}^{\bar{x}_4}\cap
\overline{H}^{\bar{x}_5}=\{\bar{e}\}.$ Consider $H\cap S(G)=K$. As $S(G)$ is
normal in $G$, we obtain that $K$ is a $\pi$-Hall subgroup of $S(G)$. In view of
\cite[Theorem~1.3]{Dolfi} or
\cite[Theorem~1.3]{VdovinIntersSolv} there exist $x,y\in S(G)$ such that $K\cap
K^x\cap K^y=O_\pi(S(G))$. As $O_\pi(G)$ is a normal $\pi$-subgroup of $G$ and
$H$ is a solvable $\pi$-Hall subgroup, we get $O_\pi(G)\leq H$ and $O_\pi(G)$
is solvable. Therefore $O_\pi(G)\leq S(G)$ and $O_\pi(G)\leq
O_\pi(S)$. Thus $O_\pi(G)=O_\pi(S)$. Therefore there exist $y_1,y_2,y_3,y_4,y_5$
such that $K^{y_1}\cap K^{y_2}\cap K^{y_3}\cap K^{y_4}\cap K^{y_5}=O_\pi(G)$.
Denote by $M_i$ the complete preimage of $\overline{H}^{\bar{x}_i}$ in $G$,
for $i=1,2,3,4,5$. Since $K^{y_i}$ and $S(G)\cap H^{x_i}$ are $\pi$-Hall
subgroup of $S(G)$ and since $S(G)$ is solvable, there exists $z_i\in S(G)$
with
\begin{equation*}
K^{y_i}=(S(G)\cap H^{x_i})^{z_i}=S(G)\cap H^{x_iz_i}.
\end{equation*}
Clearly $\overline{H}^{\bar{x}_i}=\overline{H}^{\bar{x}_i\bar{z}_i}$ and so
\begin{equation*}
 H^{x_1z_1}\cap \ldots\cap H^{x_5z_5}\subseteq S(G).
\end{equation*}
Hence $H^{x_1z_1}\cap \ldots\cap H^{x_5z_5}=K^{y_1}\cap \ldots\cap
K^{y_5}=O_\pi(G)$.\qed

\section{Final notes}

In this final section we consider two natural problems related with the main
subject of the paper.

\begin{prob}\label{LowerBound}
Given $H\leq G$, how to find a lower bound for $\Base_H(G)$?
\end{prob}

\begin{prob}\label{RegCondition}
Is it possible to remove  condition $\Reg_{S}(\Aut_G(G_i,G_{i-1}),k)\ge 5$?
\end{prob}

Consider Problem \ref{LowerBound} first. Assume that $G$ acts faithfully and
transitively on $\Omega$, and ${\Base(G)=k>1}$. Consider a $G$-regular point
$(\omega_1,\ldots,\omega_k)\in\Omega^k$. Clearly $\omega_i\not=\omega_j$ for
$i\not=j$. Hence we obtain
\begin{equation}\label{SizeGOrbit}
\vert G\vert=\vert  (\omega_1,\ldots,\omega_k)G\vert\leq
\vert\Omega\vert\cdot(\vert\Omega\vert-1)\cdot\ldots\cdot
(\vert\Omega\vert-k+1)<\vert\Omega\vert^k.
\end{equation}
Now consider $H\leq G$ such that $H$ is not normal in $G$ and assume that
$\Base_H(G)=k$. Inequality \eqref{SizeGOrbit} implies $\vert G/H_G\vert<\vert
G:H\vert^k$, and so
\begin{equation}\label{IndexBound1}
 \vert H/H_G\vert<\vert
G:H\vert^{k-1}.
\end{equation}
Inequality \eqref{IndexBound1} gives us the lower bound for $\Base_H(G)$.
Namely,
\begin{equation}\label{IndexBound2}
\Base_H(G)\ge\min\{k\mid \vert G:H\vert^{k-1}>\vert H/H_G\vert\}.
\end{equation}

Theorem~2.13 from \cite{babai} implies that there exists a constant $c$ such
that every finite group possessing a solvable subgroup of index $n$ possesses a
normal solvable subgroup of index at most $n^c$. Conjecture 6.6 from the same
paper asserts that $c\leq 7$. Therefore \eqref{IndexBound1} implies that part
(a) of Problem 17.41 from the ``Kourovka notebook'' is a strengthen of the
original Conjecture 6.6 from~\cite{babai}.

Now we discuss Problem \ref{RegCondition}. First we show that the condition
$\Reg_{S}(\Aut_G(G_i,G_{i-1}),k)\ge 5$ is
essential. The following example is given by V.I.Zenkov in~\cite{Zenkov}.

\begin{exmp}\label{S5wrS2}
 Consider $G=\Sym_5\wr \Sym_2$ and $S=\Sym_4\wr\Sym_2$. It is evident that
$\Alt_5$ is the unique nonabelian composition factor of $G$ (however there are
two nonabelian composition factors isomorphic to $\Alt_5$). It is also easy to
see, that for every solvable subgroup $T$ of $\Sym_5=\Aut(\Alt_5)$ we have
$\Base_T(\Sym_5)\le 4$.  In this case we have $\Reg_{\Sym_4}(\Sym_5,4)=1$ and
the lemma from \cite{Zenkov} implies that~$\Base_S(G)=5$.
\end{exmp}

The next example obtained in \cite{VdovinZenkov} shows that there exists an
almost simple group $G$ possessing a solvable subgroup $S$ with $\Base_S(G)=5$.

\begin{exmp}\label{S8}
 Consider $G=\Sym_8$ and $S=\Sym_4\wr\Sym_2$.  Then $\Base_S(G)=5.$ Notice also
that in \cite{VdovinZenkov} the inequality $\Reg_S(G,5)\ge12$ is
proven. Furthermore $\vert S\vert<\vert G:S\vert^2$ and so in this case
$\Base_S(G)$
is greater than the lower bound given by~\eqref{IndexBound2}.
\end{exmp}

We show that if $k\ge 6$, then we can guarantee that
$\Reg_{S}(\Aut_G(G_i/G_{i-1}),k)\ge 5$. More precisely,  the
following lemma holds.

\begin{lem}\label{LowerBoundReg}
 Let $G$ be a transitive permutation group acting on $\Omega=\{1,\ldots,n\}$ and let the stabilizer $S$ of $1$ be
solvable. Assume that $k=\max\{\Base(G),6\}$. Then $\Reg(G,k)\ge5$.
\end{lem}

We start with a technical result.

\begin{lem}\label{OrbitsAndRegularityForPointStabilizer}
Let $G$ be a transitive subgroup of $\Sym_n$. Denote $\Omega=\{1,\ldots,n\}$.
Let $H$ be the stabilizer of $1$ in $G$.
\begin{itemize}
 \item[{\em (a)}] $(1,i_2,\ldots,i_k)$ and $(1,j_2,\ldots,j_k)$  are in the same
$G$-orbit if and only if $(i_2,\ldots,i_k)$ and
$(j_2,\ldots,j_k)$ are in the same $H$-orbit;
\item[{\em (b)}] every $G$-orbit of $\Omega^k$ contains an element
$(1,i_2,\ldots,i_k)$;
\item[{\em (c)}] $(1,i_2,\ldots,i_k)$ is a $G$-regular point if and only if
$(i_2,\ldots,i_k)$ is an $H$-regular point;
\item[{\em (d)}] the number of $G$-orbits in $\Omega^k$ is equal to the number
of $H$-orbits in
$(\Omega\setminus\{1\})^{k-1}$;
\end{itemize}
\end{lem}

\begin{proof}
(a) Evident.

(b) Follows from the fact that $G$ is transitive.

(c) If $(1,i_2,\ldots,i_k)$ is a $G$-regular point, then
$(1,i_2,\ldots,i_k)g=(1,i_2,\ldots,i_k)$ implies $g=e$. Assume that $h\in H$ is
chosen so that $(i_2,\ldots,i_k)h=(i_2,\ldots,i_k)$. Since $H$ is the stabilizer
of $1$, it follows that
$(1,i_2,\ldots,i_k)h=(1,i_2,\ldots,i_k)$, hence $h=e$ and $(i_2,\ldots,i_k)$ is
an $H$-regular point. Conversely, if $(i_2,\ldots,i_k)$ is
an $H$-regular point and $(1,i_2,\ldots,i_k)g=(1,i_2,\ldots,i_k)$, we obtain
$g\in H$, and
$(i_2,\ldots,i_k)g=(i_2,\ldots,i_k)$, hence
$g=e$ and $(1,i_2,\ldots,i_k)$ is a $G$-regular point.

(d) Clear from (a), (b) and (c).
\end{proof}

\noindent {\itshape Proof of Lemma {\em\ref{LowerBoundReg}}.}\ \
 In view of Lemma \ref{OrbitsAndRegularityForPointStabilizer}, we have that $S$ acts on $\Theta=\Omega\setminus\{1\}$
and the number of $G$-regular orbits on $\Omega^k$ is equal to the number of $S$-regular orbits on $\Theta^{k-1}$. Thus
we need to prove that $\Reg(S,k-1)\ge 5$, where $S$ acts on $\Theta$. Since
$k\ge \Base(G)$, Lemma \ref{OrbitsAndRegularityForPointStabilizer} (c) implies
that  there exist
$\theta_1,\ldots,\theta_{k-1}\in\Theta$ such that $(\theta_1,\ldots,\theta_{k-1})$ is an $S$-regular point in
$\Theta^{k-1}$.

 Consider $\Delta=\{\theta_1,\ldots,\theta_{k-1}\}$, let $T$ be the setwise
stabilizer of $\Delta$
in $S$, i.e., $T=\{x\in S\mid \Delta x=\Delta\}$. It is clear that
$(\theta_{1\sigma},\ldots,\theta_{(k-1)\sigma})$ is  an $S$-regular point for every $\sigma\in\Sym_{k-1}$. Moreover if
$\sigma,\tau\in \Sym_{k-1}$, then $(\theta_{1\sigma},\ldots,\theta_{(k-1)\sigma})$ and
$(\theta_{1\tau},\ldots,\theta_{(k-1)\tau})$ are in the same $S$-orbit if and only if there exists $x\in T$ such that
$(\theta_{1\sigma},\ldots,\theta_{(k-1)\sigma})^x=(\theta_{1\tau},\ldots,\theta_{(k-1)\tau})$. Consider the restriction
homomorphism $\varphi:T\rightarrow \Sym(\Delta)$. Since $(\theta_1,\ldots,\theta_{k-1})$ is an $S$-regular point (and so
a $T$-regular point), it follows that $Ker(\varphi)=\{e\}$, i.e., $\varphi$ is injective.

Assume that $k\ge 9$ first.
Consider an asymmetric
partition $P_1\sqcup P_2\sqcup P_3\sqcup P_4\sqcup
P_5=\{\theta_1,\theta_2,\ldots,\theta_{k-1}\}$ for $T^\varphi$ (the existence
of the partition follows by Lemma \ref{Seresslemm}).
Without loss of generality we may assume that
$\vert P_1\vert\ge \vert P_2\vert\ge \vert P_3\vert\ge \vert P_4\vert\ge
\vert P_5\vert$.  Since $k\ge 9$ (and so
$\vert\{\theta_1,\theta_2,\ldots,\theta_{k-1}\}\vert\ge8$) it follows that
either
$\vert P_1\vert\ge 3$, or $\vert P_1\vert= \vert P_2\vert= \vert P_3\vert=2$.

If $\vert P_1\vert\ge 3$, then, up to renumbering, we may assume that
$\theta_1,\theta_2,\theta_3\in P_1$. In this case for every distinct
$\sigma,\tau\in \Sym_3$ we have that
$(\theta_{1\sigma},\theta_{2\sigma},\theta_{3\sigma},\theta_4\ldots,\theta_{k-1}
)$ and
$(\theta_{1\tau},\theta_{2\tau},\theta_{3\tau},\theta_4,\ldots,\theta_{k-1})$ are in distinct $T^\varphi$-orbits, thus
these points are in distinct $T$-orbits, and so in distinct $S$-orbits. So $\Reg(S,k-1)\ge \vert\Sym_3\vert=6$ in this
case.

If $\vert P_1\vert= \vert P_2\vert= \vert P_3\vert=2$, then, up to renumbering, we may assume that $\theta_1,\theta_2\in
P_1$, $\theta_3,\theta_4\in P_2$, and $\theta_5,\theta_6\in P_3$. In this case
for every distinct $\sigma,\tau\in \Sym(\{1,2\})\times
\Sym(\{3,4\})\times \Sym(\{5,6\})$ we have
that
\begin{equation*}
(\theta_{1\sigma},\theta_{2\sigma},\theta_{3\sigma},\theta_{4\sigma},\theta_{5\sigma},\theta_{6\sigma},\theta_7\ldots,
\theta_ { k-1 } )\text { and }
(\theta_{1\tau},\theta_{2\tau},\theta_{3\tau},\theta_{4\tau},\theta_{5\tau},\theta_{6\tau},\theta_7\ldots,\theta_{
k-1})
\end{equation*}
 are in distinct $T^\varphi$-orbits, thus
these points are in distinct $T$-orbits, and so in distinct $S$-orbits. So $\Reg(S,k-1)\ge \vert\Sym(\{1,2\})\times
\Sym(\{3,4\})\times \Sym(\{5,6\})\vert=8$ in this
case.

Now assume that $6\le k\le 8$. Denote by $\Xi$ the subset
$\{(\theta_{1\sigma},\ldots,\theta_{(k-1)\sigma})\mid\sigma\in \Sym_{k-1}\}$ of $\Delta^{k-1}$. Then $T^\varphi$
acts on $\Xi$ and every point of $\Xi$ is $T^\varphi$-regular. Moreover
$\vert\Xi\vert=\vert\Sym_{k-1}\vert=(k-1)!$. We also have that $T^\varphi$ is a solvable subgroup of $\Sym_{k-1}$. It
is immediate (from \cite{ATLAS}, for example), that $\vert T^\varphi\vert\le 24$
for $k=6$,
$\vert T^\varphi\vert\le 72$ for $k=7$, and $\vert T^\varphi\vert\le 144$ for
$k=8$. Now the number of
$T^\varphi$-orbits on $\Xi$ is equal to~$\frac{(k-1)!}{\vert T^\varphi\vert}$
and direct computations show that this
number is at least~$5$.
\qed

At the end of the paper we show, how $\Reg_S(G,k)$ can be applied for the
computational purposes.  If we have a group $G$
and a maximal
solvable subgroup $S$ of $G$, then Theorem \ref{MainTheorem} gives us
an idea, how to find $\Base_S(G)$, or, at least, how to find an upper bound for
$\Base_S(G)$. However, for computation
purposes it is also important  to find  the base of $G$ with respect to $S$,
i.e., elements $x_1,\ldots,x_k$ such that
$S^{x_1}\cap\ldots\cap S^{x_k}=S_G$. In general it is computationally very hard
to find these elements and we can suggest just a
probabilistic approach in this direction. Denote by $\Omega$ the set of right
cosets of $S$ in $G$. If one knows
that $\Reg_S(G,k)\ge s$ and $\vert G:S\vert=\vert\Omega\vert=n$, then
$\vert\Omega^k\vert=n^k$, while $\Omega^k$ possesses at least $s\vert
G/S_G\vert$ regular points. So the probability that $k$ randomly chosen elements
from $\Omega$ form a base of $G$ with respect to $S$ is not less than

\begin{equation*}
 \varepsilon=\frac{s\cdot \vert G/S_G\vert}{n^k}\ge \frac{s}{n^{k-1}}.
\end{equation*}

The final lemma allows to obtain a lower bound for $\Reg_S(G,k)$ in a
particular case.

\begin{lem}\label{LowerBoundForReg}
 Let $G$ be a group and let $M$ be a solvable subgroup of $\Sym_n$. Assume that
there exists $k$ such that for every maximal solvable subgroup $T$ of $G$ the
inequalities
\begin{equation*}
 \Base_T(G)\leq k\text{ and }\Reg_T(G,k)=s\ge5
\end{equation*}
hold. Then for every maximal solvable subgroup $S$ of $G\wr M$ we have
$\Reg_S(G\wr M,k)\ge s$.
\end{lem}

\begin{proof}
In the proof we preserve the notation from the proof of Lemma
\ref{WreathWithSolvableBase}. Assume that the claim is false and $G\wr M$ is a
counter example with $\vert G\wr M\vert $ minimal. Then $G\wr M$ possesses a
maximal solvable subgroup $S$ with $\Reg_S(G\wr M,k)<s$. The minimality of
$\vert G\wr M\vert$ implies that $S(G)=\{e\}$ (and so $S(G\wr M)=\{e\}$), and
$G\wr M=(G_1\times\ldots\times G_n)S$. Since $G\wr M$ is a minimal counter
example we also obtain that $M$ is transitive. Indeed, assume that $M$ is not
transitive, so $G\wr M\leq (G\wr M_1)\times (G\wr M_2)$, where $M_1\leq \Sym_m$
and $M_2\leq \Sym_{n-m}$. Up to renumbering we may suppose that $G\wr
M_1=(G_1\times \ldots\times G_m)\splitext M_1$ and $G\wr
M_2=(G_{m+1}\times\ldots\times G_n)\splitext M_2$. Denote $G_1\times
\ldots\times G_m$ by $E_1$ and $G_{m+1}\times\ldots\times G_n$ by $E_2$.
Consider the projections $\pi_1$ and $\pi_2$ of $G\wr M$ onto $G\wr M_1$ and
$G\wr M_2$ respectively. Since $G\wr M=(G_1\times\ldots\times G_n)S$, $E_1\leq
\mathrm{Ker}(\pi_2)$, and $E_2\leq \mathrm{Ker}(\pi_1)$, it follows that $(G\wr
M)\pi_i=E_i(S\pi_i)$ (we identify $E_i\pi_i$ with $E_i$ since $E_i\pi_i\simeq
E_i$). By induction for each $i\in\{1,2\}$ there
exists at least $s$  $(G\wr M)$-regular orbits with representatives
\begin{equation*}
 (Sx_{1,i,1},\ldots,Sx_{k,i,1}),\ldots,(Sx_{1,i,s},\ldots,Sx_{k,i,s}).
\end{equation*}
As we noted in the proof of Lemma \ref{WreathWithSolvableBase}, we may assume
that $x_{l,i,j}\in E_i$ for $l=1,\ldots,k$, $i=1,2$, $j=1,\ldots,s$. If we
denote
$x_{l,1,j}\cdot x_{l,2,j}$ by $x_{l,j}$, then for each $j=1,\ldots, s$ we
obtain that $(Sx_{1,j},\ldots,Sx_{k,j})$ is an $(G\wr M)$-regular point.
Clearly, $(Sx_{1,j},\ldots,Sx_{k,j})$ and $(Sx_{1,l},\ldots,Sx_{k,l})$ are in
distinct $(G\wr M)$-orbits for~${j\not=l}$.

Thus $M$ is transitive and $G\wr M=(G_1\times\ldots\times G_n)S$. Recall that
symbols $L_1,\ldots,L_n$, $\Omega_1,\ldots,\Omega_n$, $\Omega$, $\omega_{i,j}$
for $i=1,\ldots,s$, $j=1,\ldots,n$, $P_1,P_2,P_3,P_4,P_5$ are defined in the
proof of Lemma \ref{WreathWithSolvableBase}. In the proof of Lemma
\ref{WreathWithSolvableBase} we have shown that a point
$\omega=(\omega_{i_1,1},\ldots,\omega_{i_n,n})$ chosen so that
$i_t=i_j$ if and only if $t,j$ are in the same $P_l$ is an $(G\wr M)$-regular
point. If $s>5$, then for each $i=1,\ldots,s$ we can choose
$\omega^i=(\omega_{i_1,1},\ldots,\omega_{i_n,n})$ so that $i_t=i_j$
if and only if $t,j$ are in the same $P_l$ and $i\not\in \{i_1,\ldots,i_n\}$.
Now \eqref{orbitGi} implies that $\omega^1,\ldots,\omega^s$ are in distinct
$(G\wr M)$-orbits, so $G\wr M$ is not a counter example. Thus $s=5$.

Consider $\omega=(\omega_{i_1,1},\ldots,\omega_{i_n,n})$ and
$\theta=(\omega_{j_1,1},\ldots,\omega_{j_n,n})$, and assume that $\omega$ and
$\theta$ are in the same $(G\wr M)$-orbit. Therefore there exists $g=g_1\ldots
g_nh$, where $g_i\in G_i$ and $h\in M$, such that $\omega g=\theta$. The
equality $\omega g=\theta$ can be written as
\begin{equation*}
 (\omega_{i_1,1}g_1,\ldots,\omega_{i_n,n}g_n)=(\omega_{j_{(1h)},1},\ldots,
\omega_ { j_{(nh)} ,n}).
\end{equation*}
Thus for every $t=1,\ldots, n$ the equality
$\omega_{i_t,t}g_t=\omega_{j_{(th)},t}$ holds, and \eqref{orbitGi} implies that
$i_t=j_{(th)}$. Moreover, $\omega_{i,j}$ is a $G_j$-regular point for every
$i,j$, so $g_t=e$ for $t=1,\ldots,n$, i.e., $g=h\in M$. Thus we obtain that
\begin{multline}\label{criterionTheSameOrbit}
\omega\text{ and }\theta\text{ are in the same }(G\wr M)\text{-orbit  if and
only if}\\ \text{there exist }h\in M\text{ such that
}\omega_{i_t,t}=\omega_{j_{(th)},t}\text{ for }t=1,\ldots,n.
\end{multline}
Now assume that $\omega$ and $\theta$ are chosen so that
\begin{equation}\label{criterionRegularPoint}
 i_t=i_s \text{ (respectively }j_t=j_s\text{) if and only if }t,s\text{ are in
the same }P_l,
\end{equation}
in particular, $\omega$ and $\theta$ are $(G\wr M)$-regular points. If $\omega$
and $\theta$ are in the same $(G\wr M)$-orbit, then
\eqref{criterionTheSameOrbit} and \eqref{criterionRegularPoint} imply that $h$
permutes $P_1,P_2,P_3,P_4,P_5$. Since the order of a solvable subgroup of
$\Sym_5$ is not greater than $24$, we obtain that there exist at least $5$
($=\vert\Sym_5\vert/24$) points satisfying \eqref{criterionRegularPoint} and
lying in distinct $(G\wr M)$-orbits.
\end{proof}

We remark that the results in \cite{BurnessLiebeckShalev} and in the preceding
papers are obtained by using probabilistic methods. In particular, given almost
simple group $G$ with nonstandard action, it is shown that the probability for
$k$ (where $k\ge \Base(G)$) randomly chosen points to form the base tends to
$1$ as $\vert G\vert$ tends to infinity.

The author is grateful to the referee for careful consideration of the paper
and valuable comments. In particular, the referees' suggestions allows to
simplify the proof of Corollary~\ref{VdoZenkovGeneralization}.

\end{document}